\documentclass[12pt,letter, reqno]{amsart}
\usepackage{amsfonts}
\usepackage{amsthm}
\usepackage{amsmath}
\usepackage{amssymb} %
\usepackage{slashed}
\usepackage{amscd}
\usepackage[latin2]{inputenc}
\usepackage{t1enc}
\usepackage[mathscr]{eucal}
\usepackage{indentfirst}
\usepackage{graphicx}
\usepackage{graphics}
\usepackage{pict2e}
\usepackage{epic}
\numberwithin{equation}{section}
\usepackage[margin=2.9cm]{geometry}
\usepackage{epstopdf} 

\usepackage{tikz-cd}
\usepackage{dsfont}
\usetikzlibrary{matrix}

\newcommand\CC{{\mathbb C}} 
\newcommand\CP{{\mathbb C}\mathbb{P}} 
\newcommand\QQ{{\mathbb Q}} 
\newcommand\RR{{\mathbb R}} 
\newcommand\ZZ{{\mathbb Z}} 
 
\theoremstyle{plain}
\newtheorem{theorem}{Theorem}[section]
\newtheorem*{theorem*}{Theorem}
\newtheorem*{theoremA}{Theorem A}
\newtheorem*{theoremB}{Theorem B}
\newtheorem*{theoremC}{Theorem C}
\newtheorem{lemma}[theorem]{Lemma}

\newtheorem{corollary}[theorem]{Corollary}
\newtheorem{proposition}[theorem]{Proposition}

 \theoremstyle{definition}

\newtheorem{remark}[theorem]{Remark}
\newtheorem{?}[theorem]{Problem}
\newtheorem*{Question}{Question}
\newtheorem*{Answer}{Answer}

\newcommand{\im}{\operatorname{im}}

\begin{document}

\title[Connected sums of AC manifolds and products of $\QQ$-homology spheres]{Connected sums of almost complex manifolds, products of rational homology spheres, and the twisted spin${}^c$ Dirac operator}

\author{Michael Albanese and Aleksandar Milivojevi\'c}

\begin{abstract} We record an answer to the question "In which dimensions is the connected sum of two closed almost complex manifolds necessarily an almost complex manifold?". In the process of doing so, we are naturally led to ask "For which values of $\ell$ is the connected sum of $\ell$ closed almost complex manifolds necessarily an almost complex manifold?". We answer this question, along with its non-compact analogue, using obstruction theory and Yang's results on the existence of almost complex structures on $(n-1)$-connected $2n$-manifolds. Finally, we partially extend Datta and Subramanian's result on the nonexistence of almost complex structures on products of two even spheres to rational homology spheres by using the index of the twisted spin${}^c$ Dirac operator. \end{abstract}

\address{Stony Brook University \\ Department of Mathematics} 

\email{michael.albanese@stonybrook.edu}
\email{aleksandar.milivojevic@stonybrook.edu}

\maketitle

\section{Introduction}

\begin{Question}
In which dimensions is the connected sum\footnote{Note that the connected sum operation is only well-defined for connected manifolds; throughout the paper, all manifolds will be assumed to be connected.} of two closed almost complex manifolds necessarily an almost complex manifold? 
\end{Question}

The answer to this question seems to be known, but we take this opportunity to record it in detail.

In dimensions 2 and 6, we know that the connected sum of two (not necessarily closed) almost complex manifolds is almost complex. In dimension 2, every orientable manifold is almost complex, and for a smooth orientable six-manifold, the only obstruction to admitting an almost complex structure is the third integral Stiefel--Whitney class $W_3$ \cite{MasseyACS}, which is additive under connected sum. 

In dimensions $4m$, Hirzebruch used his $\chi_y$-genus to show that the Euler characteristic and signature of a closed almost complex manifold $M$ satisfy $\chi(M) \equiv (-1)^m \sigma(M) \bmod 4$ \cite[p.777]{Hirz}. Suppose we have two such manifolds $M$ and $N$ of dimension $4m$. Then $\chi(M\#N) = \chi(M) + \chi(N) - 2$ and $\sigma(M \# N) = \sigma(M) + \sigma(N)$, so we see that $\chi(M\#N) \not\equiv (-1)^m\sigma(M\#N) \bmod 4$. Therefore $M\# N$ never admits an almost complex structure. 

Unlike in the previous case, in dimensions of the form $4m + 2 > 6$, there are examples of almost complex manifolds such that their connected sum is again almost complex. For example, consider the smooth manifold $\CP^{2m+1}$. It admits an orientation-reversing diffeomorphism $[z_1, \ldots, z_{2m+2}] \mapsto [\overline{z_1}, \ldots, \overline{z_{2m+2}}]$, so there is an orientation-preserving diffeomorphism between the connected sum $\CP^{2m+1} \# \CP^{2m+1}$ and $\CP^{2m+1} \# \overline{\CP^{2m+1}}$, and the latter admits a complex structure as the blowup at a point of the standard complex structure on $\CP^{2m+1}$.

However, there are also examples of almost complex manifolds of dimension $4m + 2 > 6$ such that their connected sum does not admit an almost complex structure. For example, consider the Calabi--Eckmann complex manifold whose underlying smooth manifold is the standard $S^{2m+1}\times S^{2m+1}$. By a result of Yang \cite[Theorem 2]{Yang}, if the manifold $X = (S^{2m+1}\times S^{2m+1})\#(S^{2m+1}\times S^{2m+1})$ were to admit an almost complex structure, then we would have $(2m)! \mid \chi(X)$. As $\chi(X) = -2$, we see that $X$ does not admit an almost complex structure if $m > 1$; see also the recent \cite[Proposition 1.1]{Yang2}.

Therefore, we have an answer to the initial question.

\begin{Answer}
The connected sum of any two closed almost complex manifolds is again almost complex only in dimensions 2 and 6.
\end{Answer}

A remark of this form can be found in \cite{Geiges}, though in the references therein they ask for an almost complex structure on the connected sum which extends given almost complex structures on the summands, so the conclusion here is stronger.

In fact, Yang's result shows that the connected sum of $\ell$ copies of $S^{2m+1}\times S^{2m+1}$ does not admit an almost complex structure for $\ell \not\equiv 1 \bmod \frac{1}{2}(2m)!$. Moreover, for these highly connected manifolds, his results determine those values of $\ell$ for which the connected sum does admit an almost complex structure\footnote{In order to apply Yang's results, one must first show that the connected sum of $\ell$ copies of $S^{2m+1}\times S^{2m+1}$ is stably almost complex. This follows from Proposition 2.2 as $S^{2m+1}\times S^{2m+1}$ is stably parallelisable}. This leads us to the following question:


\begin{Question}
For which values of $\ell$ is the connected sum of $\ell$ closed almost complex manifolds necessarily an almost complex manifold?
\end{Question}

As we saw when answering the previous question, the values of $\ell$ will depend on the dimension. The argument in the case of $4m + 2 > 6$ shows that we necessarily have $\ell \equiv 1 \bmod \tfrac{1}{2}(2m)!$ in these dimensions. We remark that this can also be observed from a computation of the index of the twisted spin${}^c$ Dirac operator; see Proposition 2.9.

In Section 2, we prove the following theorem which answers the above question:

\begin{theoremA}
The connected sum of $\ell$ closed almost complex manifolds of dimension $n$ is again almost complex if
\begin{itemize}
\item $n = 4m$, and $\ell = 1$,
\item $n = 8k + 2$, and $\ell \equiv 1 \bmod (4k)!$,
\item $n = 8k + 6$, and $\ell \equiv 1 \bmod \frac{1}{2}(4k +2)!$.
\end{itemize}
For every other value of $\ell$, there exist collections of $\ell$ closed almost complex manifolds of the appropriate dimension whose connected sum does not admit an almost complex structure.
\end{theoremA}


Note that in the above theorem, $\ell = 2$ occurs only for $n = 2$ and $6$.\\

In stark contrast to the above, if the summands are allowed to be non-compact manifolds, we obtain the following (Theorem 2.5):

\begin{theoremB} The connected sum of any finite collection of almost complex manifolds of the same dimension is again almost complex if at least one of them is non-compact. \end{theoremB}

In a previous paper \cite{AM}, the twisted spin${}^c$ Dirac operator was used to extend Borel and Serre's famous result \cite{BS} on the classification of almost complex spheres to rational homology spheres. Given the prevalence of products of two spheres in the above, we apply the same technique to products of two rational homology spheres. In Section 3, we obtain the following partial generalisation of Datta and Subramanian's result \cite{DS} on the nonexistence of almost complex structures on products of two even-dimensional spheres (Theorem 3.2):

\begin{theoremC}
Let $M$ and $N$ be rational homology spheres of dimensions $2p$ and $2q$ respectively, $p, q \geq 1$.
\begin{enumerate}
\item[(a)] If $p$ and $q$ are even, then $M\times N$ does not admit an almost complex structure.
\item[(b)] If $p>1$ is odd and $q$ is even, then $M\times N$ does not admit an almost complex structure. 
\item[(c)] If $p = 1$, then $M\times N = S^2\times N$ admits an almost complex structure if and only if $q = 1$, $q = 2$, or $q = 3$ and $N$ is spin${}^c$.
\end{enumerate}
\end{theoremC}

The authors would like to thank John Morgan for a helpful discussion relating to Lemma 2.4, and Peter J. Kahn for sending us his paper \cite{Kahn} wherein the result of our Corollary 2.3 had already been noted. We would also like to thank the referee for numerous suggestions which improved the clarity of the exposition.

\section{Connected sums of almost complex manifolds}

Now we address the second question in the introduction. We will consider separately dimensions of the form $4m$ and $4m + 2$. 

\subsection{Dimensions of the form $4m$.} We show by example that there is no $\ell > 1$ for which the connected sum of any $\ell$ almost complex $4m$--manifolds is ensured to be almost complex. Consider $S^1 \times S^{4m-1}$, which admits a complex structure under which it is known as a Hopf manifold. Let $\ell\geq 2$ be arbitrary, and consider the $\ell$--fold connected sum $\#_{i=1}^{\ell}(S^1 \times S^{4m-1})$. The Pontryagin classes $p_1, \ldots, p_{m-1}$ of this manifold vanish for degree reasons, and so by the Hirzebruch signature formula $p_m$ must vanish as well, since the signature is zero and the leading coefficient in the Hirzebruch $L$--polynomial is non-zero \cite[p.12]{H}. (Alternatively, the Pontryagin classes vanish because the manifold is stably parallelisable; cf. Proposition 2.2). The Chern classes $c_1, \ldots, c_{2m-1}$ of any almost complex structure on this manifold would also vanish for degree reasons, and $c_{2m}$ would evaluate to the Euler characteristic when paired with the fundamental class. We obtain a contradiction because $0 = p_m = 2(-1)^mc_{2m}$ implies $c_{2m} = 0$, whereas the Euler characteristic is $2-2\ell \neq 0$. Therefore $\#_{i=1}^{\ell}(S^1 \times S^{4m-1})$ does not admit an almost complex structure. Note that this argument would go through for any $\ell$ closed almost complex manifolds with the same rational cohomology as a product of two odd--dimensional spheres (so in particular, for any collection of Calabi--Eckmann manifolds); although the intermediate Chern and Pontryagin may be torsion in this case, the equality $p_m = 2(-1)^m c_{2m}$ still holds.

\begin{remark} For any even number $\ell$, the connected sum of $\ell$ closed almost complex $4m$--manifolds will never admit an almost complex structure; this follows from the congruence $\chi \equiv (-1)^m \sigma \bmod 4$. For any odd number $\ell$, the connected sum $\#_{i=1}^{\ell} \CP^{2m}$ admits an almost complex structure \cite{GK}. So, there exist particular examples where a connected sum of $\ell$ closed almost complex manifolds admits an almost complex structure (provided $\ell$ is odd). \end{remark}

\subsection{Dimensions of the form $4m+2$.}

In these dimensions, we shall see that there are numbers $\ell$, depending on the dimension, which answer the above question. First we recall the following description of the stable tangent bundle of a connected sum, proven in detail in \cite[Lemma 2.1]{GK}:

\begin{proposition} Let $M$ and $N$ be oriented smooth manifolds of dimension $d$. Denote by $p_M$ and $p_N$ the collapsing maps $M\#N \to M$ and $M\# N \to N$. Then, as oriented real vector bundles, $T(M\# N) \oplus \varepsilon_{\RR}^{d} \cong p_M^*(TM) \oplus p_N^*(TN)$. \end{proposition}

We will make use of the following observation:

\begin{corollary} Let $M$ and $N$ be oriented smooth manifolds of dimension $d$. If $M$ and $N$ admit stable almost complex structures $\omega_M$ and $\omega_N$, then $M\# N$ admits a stable almost complex structure $\omega_{M\#N}$ such that $$c_j(\omega_{M\#N}) = p_M^*(c_j(\omega_M)) + p_N^*(c_j(\omega_N))$$ for all $j \geq 1$.\end{corollary}

\begin{proof} Since $M$ and $N$ are stably almost complex, there are integers $k$ and $l$ such that $TM\oplus \varepsilon_{\RR}^k$ and $TN\oplus \varepsilon_{\RR}^l$ admit the structure of complex vector bundles. Then $p_M^*(TM \oplus \varepsilon_{\RR}^k) \oplus p_N^*(TN \oplus \varepsilon_{\RR}^l)$ admits the structure of a complex vector bundle, and as oriented real vector bundles we have $$p_M^*(TM \oplus \varepsilon_{\RR}^k) \oplus p_N^*(TN \oplus \varepsilon_{\RR}^l) \cong p_M^*(TM) \oplus p_N^*(TN) \oplus \varepsilon_{\RR}^{k+l} \cong T(M\# N) \oplus \varepsilon_{\RR}^{k+l+d}.$$ Therefore $T(M\# N) \oplus \varepsilon_{\RR}^{k+l+d}$ admits the structure of a complex vector bundle which is isomorphic to $p_M^*(TM \oplus \varepsilon_{\RR}^k) \oplus p_N^*(TN \oplus \varepsilon_{\RR}^l)$ as complex vector bundles. Let us denote this stable almost complex structure on $M\#N$ by $\omega_{M\#N}$.

Now by naturality and the Cartan formula for Chern classes, we have the following for any $j\geq 1$: \begin{align*}  c_j(\omega_{M \# N}) &= c_j(T(M\# N) \oplus \varepsilon_{\RR}^{k+l+d}) = c_j(p_M^*(TM \oplus \varepsilon_{\RR}^k) \oplus p_N^*(TN \oplus \varepsilon_{\RR}^l)) \\ &= \sum_{i=0}^j c_i(p_M^*(TM \oplus \varepsilon_{\RR}^k)) \, c_{j-i}(p_N^*(TN \oplus \varepsilon_{\RR}^l)) \\ &= \sum_{i=0}^j p_M^*(c_i(TM \oplus \varepsilon_{\RR}^k)) \, p_N^*(c_{j-i}(TN \oplus \varepsilon_{\RR}^l)) \\ &= p_M^*(c_j(TM \oplus \varepsilon_{\RR}^k)) + p_N^*(c_j(TN \oplus \varepsilon_{\RR}^l)).\end{align*} In the last equality, we used that in degrees strictly between $0$ and $d$ we have $\im{p_M^*} \cdot \im{p_N^*} = 0$.   \end{proof}

Now we note that if we have a stable almost complex structure $\omega$ on a closed manifold $M$, then there exists an almost complex structure on the non-compact manifold $M \setminus D$ which induces a stable almost complex structure isomorphic to $\omega \vert_{M\setminus D}$ upon stabilisation, and extends to a stable almost complex structure over $M$ which is isomorphic to $\omega$. Here $D$ denotes an embedded closed disc of the same dimension as $M$. This is remarked on in \cite[p. 345, paragraph 2]{K} (see also \cite{Kahn}), but we include the details for the convenience of the reader. 

\begin{lemma} On a non-compact $2n$--manifold $X$ with stable almost complex structure $\omega_X$, there is an almost complex structure $J$ such that the map $X \overset{J}{\longrightarrow} BU(n) \rightarrow BU$ is homotopic to $X \overset{\omega_X}{\longrightarrow}BU$ (i.e. $J$ induces a stable almost complex structure isomorphic to $\omega_X$). \end{lemma}

\begin{proof} Since $X$ has the homotopy type of a finite-dimensional cell complex, the stable almost complex structure $\omega_X$ on $TX$ is given by a lifting of the composition $X {\rightarrow} BSO(2n) \to BSO(2n+2N)$ through the map $BU(n+N) \to BSO(2n+2N)$ for some $N$ (corresponding to $TX \oplus \varepsilon_{\RR}^{2N}$ being a complex bundle), where the map $X \to BSO(2n)$ represents the real tangent bundle $TX$. Consider the following diagram (we label all the mentioned maps for ease of reference):

$$ \begin{tikzcd}[column sep=small, row sep = large] & & & & & BU(n) \arrow[r, "i_U"] \arrow[d, "j_n"] & BU(n+N) \arrow[d, "j"] \\  X \arrow[rrrrr, "\phi"] \arrow[rrrrru, dashed, "J"] \arrow[rrrrrru, "\omega_X"] & & & & & BSO(2n) \arrow[r, "i_{SO}"] & BSO(2n+2N) \end{tikzcd} $$

We wish to find a map $X \overset{J}{\to} BU(n)$ such that the composition $X \overset{J}{\to} BU(n) \overset{j_n}{\longrightarrow} BSO(2n)$ is homotopic to $X \overset{\phi}{\to} BSO(2n)$ and such that $X \overset{J}{\to} BU(n) \overset{i_U}{\longrightarrow} BU(n+N)$ is homotopic to $X \overset{\omega_X}{\longrightarrow} BU(n+N)$. Since the homotopy fiber of $BU(n) \overset{i_U}{\longrightarrow} BU(n+N)$ has the homotopy type of the Stiefel manifold of complex $N$-frames in $\CC^{n+N}$, which is $2n$--connected, the map $i_U$ induces a bijection on homotopy classes of maps $[X, BU(n)] \rightarrow [X, BU(n+N)]$ and so we have a (unique) lift $J$ of $\omega_X$ through $i_U$. 

Now we check that $J$ is a lift of $\phi$, i.e. $j_n J \sim \phi$. Since $i_U J \sim \omega_X$, we have $ji_U J \sim j\omega_X$. Therefore, because $i_{SO} j_n = ji_U$, we have $i_{SO}j_n J = ji_UJ \sim j\omega_X = i_{SO}\phi$, i.e. the two maps $j_n J$ and $\phi$ are lifts of the same map (up to homotopy) $X \to BSO$ through the map $BSO(2n) \overset{i_{SO}}{\longrightarrow} BSO(2n+2N)$. The homotopy fiber of $BSO(2n) \overset{i_{SO}}{\longrightarrow} BSO(2n+2N)$ has the homotopy type of the Stiefel manifold of real $2N$-frames in $\RR^{2n+2N}$, which is $(2n-1)$--connected. Since $X$ has the homotopy type of a $(2n-1)$--dimensional complex \cite[Lemma 2.1]{Whitehead}, $i_{SO}$ induces a bijection $[X, BSO(2n)] \rightarrow [X, BSO(2n+2N)]$,  and so $j_n J$ and $\phi$ are homotopic.
\end{proof}

Now the claim prior to the lemma follows by taking $X = M\setminus D$ and $\omega_X = \omega \vert_{M\setminus D}$. Indeed, we obtain an almost complex structure $J$ on $M\setminus D$ such that $M\setminus D \overset{i}{\hookrightarrow} M \overset{\omega}{\rightarrow} BU$ is homotopic to $M\setminus D \overset{J}{\rightarrow} BU(n) \overset{i_U}{\longrightarrow} BU$, and so applying the homotopy extension property to the inclusion of $M\setminus D$ into $M$ yields the desired extension.

We also observe that the above lemma, combined with the fact that the connected sum of stably almost complex manifolds is stably almost complex, gives us the following observation:

\begin{theorem} The connected sum of any finite collection of almost complex manifolds of the same dimension is again almost complex if at least one of them is non-compact. \end{theorem}

Returning to the compact case, we consider separately dimensions of the form $8k + 2$ and $8k + 6$.

\begin{theorem}  Let $M_i$, $1\leq i \leq \ell$, be closed almost complex manifolds of real dimension $8k+2$. If $\ell \equiv 1 \bmod (4k)!$, then $\#_{i=1}^{\ell} M_i$ admits an almost complex structure. \end{theorem}

\begin{proof} First, by Corollary 2.3, $M := \#_{i=1}^{\ell} M_i$ admits a stable almost complex structure whose top Chern class $c_{4k+1}$ is the sum of the corresponding top Chern classes on the summands. By the above, there is an almost complex structure $J$ on $M\setminus D$ which extends to a stable almost complex structure on $M$ which is isomorphic to $\omega$. If $\mathfrak{o}(M, J) \in H^{8k+2}(M; \pi_{8k+1}(SO(8k+2)/U(4k+1))) \cong \pi_{8k+1}(SO(8k+2)/U(4k+1))$ denotes the obstruction to extending $J$ to an almost complex structure on $M$, then by \cite[Corollary 2]{K}, see also \cite[Lemma 2]{Geiges}, we have $$\mathfrak{o}(M, J) = \tfrac{1}{2}(\chi(M) - c_{4k+1}(\omega))\mathfrak{o}(S^{8k+2}) \in \pi_{8k+1}(SO(8k+2)/U(4k+1)).$$ Here $\mathfrak{o}(S^{8k+2})$ denotes the obstruction to extending an almost complex structure on $S^{8k+2}\setminus D$ to all of $S^{8k+2}$; this obstruction does not depend on the almost complex structure \cite[p. 339]{K} or \cite[Lemma 2]{Geiges}. As $\pi_{8k+1}(SO(8k+2)/U(4k+1)) \cong \ZZ_{(4k)!}$ \cite[Table 3]{K}, we see that $\mathfrak{o}(M, J) = 0$ if $\tfrac{1}{2} (\chi(M) - c_{4k+1}(\omega)) = \tfrac{1}{2} (\chi(M) - \sum_{i=1}^{\ell} \chi(M_i)) = \tfrac{1}{2} (-2(\ell-1)) = 1 - \ell \equiv 0 \bmod (4k)!$. Therefore, this almost complex structure on $M\setminus D$ extends to $M$ if $\ell \equiv 1 \bmod (4k)!$. 

\end{proof}

\begin{theorem}  Let $M_i$, $1\leq i \leq \ell$, be closed almost complex manifolds of real dimension $8k+6$. If $\ell \equiv 1 \bmod \tfrac{1}{2}(4k+2)!$, then $\#_{i=1}^{\ell} M_i$ admits an almost complex structure. \end{theorem}

\begin{proof} The proof is the same as in the previous theorem, with the adjustment that $\pi_{8k+5}(SO(8k+6)/U(4k+3)) \cong \ZZ_{(4k+2)!/2}$ \cite[Table 3]{K}. \end{proof}

\begin{remark}By \cite[Lemma 8]{K}, the element $\mathfrak{o}(S^{2n})$ generates the kernel of the map $\pi_{2n-1}(SO(2n)/U(n)) \rightarrow \pi_{2n-1}(SO/U)$ induced by the inclusion $SO(2n)/U(n) \hookrightarrow SO/U$. As $\pi_{8k+1}(SO/U) = \pi_{8k+5}(SO/U) = 0$ \cite[Table 3]{K}, we see that the stable almost complex structure on $\#_{i=1}^{\ell}M_i$ given by Corollary 2.3 is induced by an almost complex structure if and only if $\ell$ is of the appropriate above form.\end{remark}

We now observe that the numbers $\ell$ obtained in the above theorems are optimal, in the following sense: We can find examples of closed almost complex manifolds $M_i$ such that $\#_{i=1}^{\ell} M_i$ admits an almost complex structure if and only if $\ell$ is of the above form in the appropriate dimension. Recall that a product of two odd-dimensional spheres admits an (almost) complex structure. Applying \cite[Theorem 2]{Yang}, we see that the connected sum $\#_{i=1}^{\ell} (S^{4k+1} \times S^{4k+1})$ admits an almost complex structure if and only if $\ell \equiv 1 \bmod (4k)!$ and $\#_{i=1}^{\ell}(S^{4k+3} \times S^{4k+3})$ admits an almost complex structure if and only if $\ell \equiv 1 \bmod \tfrac{1}{2}(4k+2)!$.


We can easily produce additional examples of collections of $(4m+2)$--dimensional closed almost complex manifolds for which the connected sum is not almost complex for $\ell \not \equiv 1 \bmod \tfrac{1}{2}(2m)!$. When $m$ is odd, these are further examples of collections where the values of $\ell$ provided by Theorem 2.7 are optimal.

\begin{proposition} Let $M_i$, $i=1, \ldots, \ell$, be a collection of $(4m+2)$--dimensional closed almost complex manifolds with $H^{2j}(M_i;\QQ) = 0$ for $1\leq j \leq 2m$. If the connected sum $M:= \#_{i=1}^{\ell} M_i$ admits an almost complex structure, then $\ell \equiv 1 \bmod \tfrac{1}{2}(2m)!$. \end{proposition}

\begin{proof} 
If $M$ were to admit an almost complex structure, then the Chern classes $c_j(TM)$, $1\leq j \leq 2m$ would be torsion by the assumption on the cohomology groups. As for the top Chern class $c_{2m+1}(TM) \in H^{4m+2}(M ; \ZZ) \cong \ZZ$, it integrates to the Euler characteristic, which is $\sum_{i=1}^{\ell} \chi(M_i) - 2(\ell-1)$. Likewise we see that all the Pontryagin classes of the underlying smooth manifold would have to be torsion. Since an almost complex structure induces a canonical spin${}^c$ structure, we can form the twisted spin${}^c$ Dirac operator $\slashed{\partial}^c_{TM}$ which has index (see \cite[Theorem 26.1.1]{H})

$$\operatorname{ind}(\slashed{\partial}^c_{TM}) = \int_M \exp(c_1(TM)/2) ch(TM) \hat{A}(TM).$$

Since the Pontryagin classes and $c_1(TM), \ldots, c_{2m}(TM)$ are torsion, we see that, modulo torsion, $\hat{A}(TM) = 1$, $\exp(c_1(TM)/2) = 1$ and $ch(TM) = \tfrac{1}{2}\dim M + \tfrac{1}{(2m)!} c_{2m+1}(TM)$. Therefore, $\operatorname{ind}(\slashed{\partial}^c_{TM}) = \tfrac{1}{(2m)!}\chi(M) = \tfrac{1}{(2m)!}(\sum_{i=1}^{\ell} \chi(M_i) - 2(\ell-1)) \in \mathbb{Z}$. The same argument applied to each $M_i$ shows that $\tfrac{1}{(2m)!}\chi(M_i) \in \mathbb{Z}$. Therefore $\ell - 1$ must be divisible by $\tfrac{1}{2}(2m)!$.
\end{proof}


\section{Almost complex structures on the product of two rational homology spheres}

Borel and Serre \cite{BS} determined which spheres admit almost complex structures:

\begin{theorem*}
The only spheres which admit almost complex structures are $S^2$ and $S^6$.
\end{theorem*}

Using the index of the twisted spin${}^c$ Dirac operator, the current authors were able to answer the analogous question for rational homology spheres \cite[Theorem 2.2]{AM}.

\begin{theorem*}
A rational homology sphere $M$ admits an almost complex structure if and only if $\dim M = 2$ (i.e. $M = S^2$), or $\dim M = 6$ and $M$ is spin${}^c$.
\end{theorem*}

Determining which products of two spheres admit an almost complex structure separates into two cases: when both spheres are odd-dimensional, and when both spheres are even-dimensional. In the first case, almost complex structures always exist because the manifolds are parallelisable; in fact, such manifolds admit integrable almost complex structures, e.g. Hopf manifolds and Calabi-Eckmann manifolds. The second case was addressed by Datta and Subramanian in \cite{DS}.

\begin{theorem}
The only products of even-dimensional spheres which admit almost complex structures are $S^2\times S^2$, $S^2\times S^4$, $S^2\times S^6$, and $S^6\times S^6$.
\end{theorem}

In the paper, the authors remark that their result only applies to the standard smooth structures on spheres as they use the stable triviality of the tangent bundle. However, exotic spheres also have stably trivial tangent bundle as shown by Kervaire and Milnor \cite[Theorem 3.1]{KM}, so the theorem can be extended to allow any smooth structures on the factor spheres. The spheres $S^2$ and $S^6$ have a unique smooth structure, while it is a famous open problem as to whether or not $S^4$ admits exotic smooth structures. However, for any hypothetical choice of smooth structure on $S^4$, the manifold $S^2\times S^4$ admits an almost complex structure as it is a spin${}^c$ six--manifold. Note, it's possible that there is a product of even-dimensional topological spheres not on the above list which admits a non-product smooth structure for which an almost complex structure exists.
 
The twisted spin${}^c$ Dirac operator can be used to address the rational analogue of this question, namely: Which products of two rational homology spheres admit almost complex structures? As was shown in \cite{AM}, there are non-spin${}^c$ rational homology spheres in every dimension greater than four, so for most choices of dimensions $m \leq n$, there are examples of rational homology spheres $M$ and $N$ such that $M\times N$ is not spin${}^c$ and hence does not admit an almost complex structure. In fact, the only cases of $m$ and $n$ where any choice of rational homology spheres gives a product which admits an almost complex structure are $(m, n) = (1, 1), (1, 3), (2, 2)$, and $(2, 4)$. So, a natural question to ask is: For which dimensions $m$ and $n$ do there \textit{exist} rational homology spheres $M$ and $N$ with $M\times N$ almost complex? As before, this separates into two cases depending on the parity of the dimensions.

If both dimensions are odd, there are always examples, namely spheres themselves. In fact, as integral homology spheres are stably parallelisable (see \cite[p.70]{Kervaire}, and \cite{ORW} for a detailed explanation), the product of two odd-dimensional integral homology spheres is parallelisable, see e.g. \cite{Staples}, and hence admits an almost complex structure. 

If both dimensions are even, we have the following (partial) result:

\begin{theorem}
Let $M$ and $N$ be rational homology spheres of dimensions $2p$ and $2q$ respectively, $p, q \geq 1$.
\begin{enumerate}
\item[(a)] If $p$ and $q$ are even, then $M\times N$ does not admit an almost complex structure.
\item[(b)] If $p>1$ is odd and $q$ is even, then $M\times N$ does not admit an almost complex structure. 
\item[(c)] If $p = 1$, then $M\times N = S^2\times N$ admits an almost complex structure if and only if $q = 1$, $q = 2$, or $q = 3$ and $N$ is spin${}^c$.
\end{enumerate}
\end{theorem}

\begin{proof}
Note that the Pontryagin classes of rational homology spheres are all torsion (in dimensions divisible by four, the top Pontryagin class vanishes due to the signature being zero as argued in Section 2.1). Therefore all the Pontryagin classes of $M\times N$ are torsion, so if $M\times N$ is almost complex, then in $H := H^*(M\times N; \mathbb{Z})/H^*(M\times N; \mathbb{Z})_{\text{tors}}$ we have the equality $$1 = c(T(M\times N))c(\overline{T(M\times N)}).$$ 

Following \cite{DS}, let $u$ and $v$ be generators in $H$ of degree $2p$ and $2q$ respectively such that $c_{p+q}(T(M\times N)) = 4uv$. Then in $H$ we have $c(T(M\times N)) = 1 + au + bv + 4uv$ for some $a, b \in \mathbb{Z}$ and hence $c(\overline{T(M\times N)}) = 1 + (-1)^pau + (-1)^qbv + 4(-1)^{p+q}uv$ in $H$. 

So we see that in $H$ we have
\begin{align}
1 &= 1 + [1 + (-1)^p]au + [1+(-1)^q]bv + [4(-1)^{p+q} + (-1)^qab + (-1)^pab + 4]uv. \label{eq1}
\end{align}
It follows that if $p$ is even, then $a = 0$, and if $q$ is even, then $b = 0$. So if $p$ and $q$ are both even, we arrive at a contradiction as the coefficient of $uv$ is $8 \neq 0$; this establishes case (a).

In what follows, we will need to make use of the Chern character of $T(M\times N)$ which takes the form (in $H$) $$\operatorname{ch}(T(M\times N)) = (p + q) + \frac{s_p}{p!} + \frac{s_q}{q!} + \frac{s_{p+q}}{(p+q)!}$$ where the classes $s_i$ are determined by Newton's relation between elementary symmetric polynomials and power sums:
\begin{align}
s_n - c_1s_{n-1} + c_2s_{n-2} + \dots + (-1)^nnc_n &= 0.
\end{align}

If $p$ is odd and $q$ is even, then we have $b = 0$. For $p > 1$, the class $c_1(T(M\times N))$ is torsion, and hence, recalling that the Pontryagin classes of $T(M\times N)$ are also torsion, we have
\begin{align*}
\operatorname{ind}(\slashed{\partial}^c_{T(M\times N)}) &= \int_{M\times N}\operatorname{exp}(c_1(T(M\times N))/2)\operatorname{ch}(T(M\times N))\hat{A}(T(M\times N))\\ 
&= \int_{M\times N}\operatorname{ch}_{p+q}(T(M\times N))\ = \int_{M\times N}\frac{c_{p+q}}{(p + q - 1)!}  \\ &= \int_{M\times N}\frac{4}{(p + q - 1)!} uv
\end{align*}
where the penultimate equality follows from $(3.2)$. So $(p + q - 1)! \mid 4$, but as $p>1$ is odd and $q$ is even, we see that this is impossible; this establishes case (b).

Finally, we deal with case (c). If $q=1$, then $M\times N = S^2 \times S^2$; so suppose $q>1$. Modulo torsion, $c_1(T(M\times N)) = au$ for some $a \in \mathbb{Z}$. If $q$ is even, then $b = 0$, so
\begin{align*}
\operatorname{ind}(\slashed{\partial}^c_{T(M\times N)}) &= \int_{M\times N}\operatorname{exp}(c_1(T(M\times N))/2)\operatorname{ch}(T(M\times N))\hat{A}(T(M\times N))\\ 
&= \int_{M\times N}\left(1 + \frac{1}{2}au\right)\left(1+q + au + \frac{4}{q!}uv\right) = \int_{M\times N}\frac{4}{q!}uv
\end{align*}
where again we have used $(3.2)$ to establish the penultimate equality. Hence $q! \mid 4$, which is only possible for $q=2$. Since an orientable four--manifold is spin${}^c$, the product $S^2\times N$ is a spin${}^c$ six--manifold and hence almost complex for any four--dimensional rational homology sphere $N$.

Suppose now that $q>1$ is odd. Then, again using $(3.2)$, we have
\begin{align*}
\operatorname{ind}(\slashed{\partial}^c_{T(M\times N)}) &= \int_{M\times N}\operatorname{exp}(c_1(T(M\times N))/2)\operatorname{ch}(T(M\times N))\hat{A}(T(M\times N))\\ 
&= \int_{M\times N}\left(1 + \frac{1}{2}au\right)\left((1 + q) + au + \frac{b}{(q-1)!}v + \frac{4-ab}{q!}uv \right)\\
&= \int_{M\times N}\left( \frac{ab}{2(q-1)!}+\frac{4-ab}{q!}\right)uv
\end{align*}

From (\ref{eq1}), we see that $ab = 4$ and hence $(q-1)! \mid 2$, so $q=3$. As $S^2 \times N$ is almost complex, $N$ must be spin${}^c$; conversely, a six--dimensional spin${}^c$ manifold is almost complex and hence $S^2 \times N$ is almost complex for any six--dimensional spin${}^c$ rational homology sphere $N$. 
\end{proof}

\begin{remark}
It is worth noting that part (a) of the above theorem can be generalised to arbitrary products by the same argument. Alternatively, the fact that the Pontryagin classes and odd Chern classes are torsion implies that the even Chern classes are torsion, but this contradicts the fact that the top Chern class is the Euler class which is non-zero.
\end{remark}

The argument in \cite{DS} uses the fact that $\operatorname{ch}(E) \in H^{\text{even}}(S^{2p}\times S^{2q}; \mathbb{Z})$ for any complex vector bundle $E \to S^{2p}\times S^{2q}$, in particular $E = T(S^{2p}\times S^{2q})$. As in the proof above, one can deduce that $\operatorname{ch}_{p+q}(E) \in H^{2(p + q)}(S^{2p}\times S^{2q}; \mathbb{Z})$ from the index of the twisted spin${}^c$ Dirac operator $\slashed{\partial}^c_E$. However, the index theorem doesn't show that $\operatorname{ch}_p(E) \in H^{2p}(S^{2p}\times S^{2q}; \mathbb{Z})$ and $\operatorname{ch}_q(E) \in H^{2q}(S^{2p}\times S^{2q}; \mathbb{Z})$. These integrality statements are precisely what Datta and Subramanian exploit for the case we are missing, namely for odd $p,q>1$, in which case the above index is zero.

We conclude with the open question whether the following rational analogue of Datta and Subramanian's theorem is true:

\begin{?}
Let $M$ and $N$ be rational homology spheres with $\dim M = 2p$ and $\dim N = 2q$ where $p$ and $q$ are odd. If $M\times N$ admits an almost complex structure, is it necessarily the case that $\dim M, \dim N \in \{2, 6\}$?
\end{?}

\end{document}